\documentclass{amsart}
\usepackage{amssymb,amstext,amsmath,amscd,amsthm,amsfonts,enumerate,graphicx,latexsym,stmaryrd,multicol,bm}

\usepackage[usenames]{color}
\usepackage[all]{xy}
\newtheorem{thm}{Theorem}[section]
\newtheorem{lem}[thm]{Lemma}
\newtheorem{prop}[thm]{Proposition}
\newtheorem{cor}[thm]{Corollary}
\theoremstyle{definition}
\newtheorem{dfn}[thm]{Definition}

\newtheorem{rmk}[thm]{Remark}

\theoremstyle{remark}

\newtheorem*{ac}{Acknowlegments}
\newtheorem*{conv}{Convention}

\numberwithin{equation}{thm}
\def\A{\mathcal{A}}
\def\add{\operatorname{add}}
\def\Add{\operatorname{Add}}
\def\B{\mathcal{B}}
\def\bound{\operatorname{B}}
\def\C{\mathrm{C}}
\def\cobound{\operatorname{C}}
\def\CM{\operatorname{CM}}
\def\coker{\operatorname{Coker}}
\def\cone{\operatorname{cone}}
\def\coresoldim{\operatorname{-coresol}.\operatorname{dim}}
\def\cycle{\operatorname{Z}}
\def\D{\mathrm{D}}
\def\depth{\operatorname{depth}}
\def\dim{\operatorname{dim}}
\def\End{\operatorname{End}}
\def\Ext{\operatorname{Ext}}
\def\G{\operatorname{G}}
\def\Gdim{\operatorname{Gdim}}
\def\Gid{\operatorname{Gid}}
\def\Ginj{\operatorname{Ginj}}
\def\GInj{\operatorname{GInj}}
\def\Gpd{\operatorname{Gpd}}
\def\Gproj{\operatorname{Gproj}}
\def\GProj{\operatorname{GProj}}
\def\Hom{\operatorname{Hom}}
\def\homol{\operatorname{H}}
\def\id{\operatorname{id}}
\def\inf{\operatorname{inf}}
\def\inj{\operatorname{inj}}
\def\Inj{\operatorname{Inj}}
\def\image{\operatorname{Im}}
\def\K{\mathrm{K}}
\def\ker{\operatorname{Ker}}

\def\level{\operatorname{level}}
\def\Lotimes{\otimes^{\mathbf{L}}}

\def\mod{\operatorname{mod}}
\def\Mod{\operatorname{Mod}}
\def\pd{\operatorname{pd}}
\def\proj{\operatorname{proj}}
\def\Proj{\operatorname{Proj}}
\def\RHom{\operatorname{\mathbf{R}Hom}}
\def\Sub{\operatorname{Sub}}
\def\sup{\operatorname{sup}}
\def\T{\mathcal{T}}
\def\thick{\operatorname{thick}}
\def\tref{\operatorname{tref}}

\def\X{\mathcal{X}}

\def\resoldim{\operatorname{-resol}.\operatorname{dim}}

\def\Y{\mathcal{Y}}
\def\ZZ{\mathbb{Z}}

\begin{document}
\allowdisplaybreaks
\title[Lower bounds for levels of complexes by resolution dimensions]{Lower bounds for levels of complexes\\by resolution dimensions}
\author{Yuki Mifune}
\address{Graduate School of Mathematics, Nagoya University, Furocho, Chikusaku, Nagoya 464-8602, Japan}
\email{yuki.mifune.c9@math.nagoya-u.ac.jp}
\thanks{2020 {\em Mathematics Subject Classification.} 13C60, 13D09, 18G80}
\thanks{{\em Key words and phrases.} contravariantly finite subcategory, derived category, ghost lemma, Gorenstein projective, level, resolving subcategory, orthogonal subcategory, semidualizing module}
\begin{abstract}
Let $\mathcal{A}$ be an abelian category. Denote by $\mathrm{D}^{b}(\mathcal{A})$ the bounded derived category of $\mathcal{A}$. In this paper, we investigate the lower bounds for the levels of objects in $\mathrm{D}^{b}(\mathcal{A})$ with respect to a (co)resolving subcategory satisfying a certain condition. As an application, we not only recover the results of Altmann--Grifo--Monta\~{n}o--Sanders--Vu, and Awadalla--Marley but also extend them to establish lower bounds for levels with respect to some other subcategories in an abelian category. 
\end{abstract}
\maketitle
\section{Introduction}
The notion of the level of an object in a triangulated category was explicitly introduced by Avramov, Buchweitz, Iyengar, and Miller \cite{ABIM}. It measures the number of steps of taking mapping cones required to construct a given object from a specified full subcategory. Prior to this formal introduction, similar concepts had been considered; see \cite{BV, JD Christensen, DGI, KK, Rouquier} for instance.

Let $\A$ be an abelian category, and denote by $\D^{b}(\A)$ the bounded derived category of $\A$. An upper bound for levels of objects in $\D^{b}(\A)$ can often be easily computed by using hard truncation (\cite[Corollary 2.7]{Awadalla Marley} for example). This paper, however, focuses on lower bounds. In particular, we study the level of a complex with respect to the essential images in $\D^{b}(\A)$ of specific full subcategories of $\A$. 
Let $\X$ be a full subcategory of $\A$. By definition, the {\em $\X$-resolution dimension} of a complex $M$ (see Definition \ref{def of resoldim}) is always more than or equal to the negative of the infimum of the integers $i$ such that the $i$-th cohomology $\homol^{i}\!M$ does not vanish. 
In the previous researches, the gap between these two values has been used to establish a lower bound for the levels of $M$ with respect to $\X$ in the case where $\A$ is the category of modules over a ring and $\X$ is the category of (Gorenstein) projective modules.
\begin{thm}[Altmann, Grifo, Monta\~{n}o, Sanders, and Vu]\label{AGMSV int}
Let $R$ be a ring, and $M$ a nonzero complex of left $R$-modules. Then one has
\begin{center}
$\level_{\D^{b}(\Mod R)}^{\Proj R}M \geq \pd_{R}M + \inf M +1$.
\end{center}
\end{thm}
\begin{thm}[Awadalla and Marley]\label{AM int}
Let $R$ be a commutative ring.
Then for all nonzero objects $M$ in $\D^{b}(\Mod R)$, one has
\begin{center}
$\level_{\D^{b}(\Mod R)}^{\GProj R}M \geq \Gpd_{R}M + \inf M +1$.
\end{center}
Moreover, if $R$ is noetherian and $M$ is in $\D^{b}(\mod R)$, then one has
\begin{center}
$\level_{\D^{b}(\mod R)}^{\tref(R)}M \geq \Gdim_{R}M + \inf M +1$.
\end{center}
\end{thm}
When considering a lower bound for the levels of objects in $\D(\A)$ with respect to $\X$, we may assume that they belong to $\D^{b}(\A)$.
For the notation used in the above results, see Definitions \ref{def of Gproj} and \ref{def of Gcproj}.
The main result of this paper is the following theorem.
\begin{thm}[Theorem \ref{main thm llevel}]\label{main thm intro}
Let $\A$ be an abelian category with enough projective objects, and $\X$ a resolving  subcategory of $\A$. Suppose that for every $X\in \X$, there exists an exact sequence $0\to X\to Y\to X'\to 0$ with $Y\in\X\cap\X^{\perp}$, and $X'\in\X$.
Then for all nonzero objects $M$ in $\D^{b}(\A)$, the following inequality holds.
\begin{center}
$\level_{\D^{b}(\A)}^{\X}M 
\geq \X \resoldim M + \inf M +1$.
\end{center}
\end{thm}
As an application of the above theorem, we obtain lower bounds for the levels of complexes with respect to various subcategories of an abelian category, which extends Theorems \ref{AGMSV int} and \ref{AM int} due to Altmann--Grifo--Monta\~{n}o--Sanders--Vu, and Awadalla--Marley.
\begin{cor}[Corollaries \ref{Gproj inj}, \ref{Gcproj inj}, \ref{C level}, and \ref{cor of AR}]
Let $\A$ and $\B$ be abelian categories with $\A$ having enough projective objects and $\B$ having enough injective objects. Let $R$ be a commutative noetherian ring, and $C$ a semidualizing $R$-module. Let $\Lambda$ be an artin algebra, and $\X$ a contravariantly finite resolving subcategory of $\mod \Lambda$. Then for all nonzero objects $X\in \D^{b}(\A)$, $Y\in \D^{b}(\B)$, $M\in\D^{b}(\Mod R)$, $N\in\D^{b}(\mod R)$, and $L\in\D^{b}(\mod \Lambda)$, we have the following inequalities.
\begin{center}
$\level_{\D^{b}(\A)}^{\proj\A}X\geq \pd_{\A}X+\inf X +1$,
\quad
$\level_{\D^{b}(\B)}^{\inj\B}Y\geq \id_{\B}Y-\sup Y +1$,
\vskip.5\baselineskip
$\level_{\D^{b}(\A)}^{\Gproj\A}X\geq \Gpd_{\A}X+\inf X +1$,
\quad 
$\level_{\D^{b}(\B)}^{\Ginj\B}Y\geq \Gid_{\B}Y-\sup Y +1$,
\vskip.5\baselineskip
$\level_{\D^{b}(\Mod R)}^{\G_{C}\Proj R}M\geq \G_{C}\pd_{R}M+\inf M +1$,
\quad
$\level_{\D^{b}(\Mod R)}^{\G_{C}\Inj R}M\geq \G_{C}\id_{R}M-\sup M +1$,
\vskip.5\baselineskip
$\level_{\D^{b}(\mod R)}^{\tref_{C}(R)}N\geq \G_{C}\dim_{R}N+\inf N +1$,
\vskip.5\baselineskip
$\level_{\D^{b}(\Mod R)}^{\Add C}M\geq \Add C\resoldim M +\inf M +1$, 
\vskip.5\baselineskip
$\level_{\D^{b}(\mod R)}^{\add C}N\geq \add C\resoldim N +\inf N +1$,
\vskip.5\baselineskip
$\level_{\D^{b}(\mod \Lambda)}^{\X}L \geq \X \resoldim L + \inf L +1$.
\end{center}
\end{cor}
The organization of this paper is as follows. In section 2, we provide several fundamental definitions and terminology necessary for proving our main theorem.
In Section 3, we first state and prove various lemmas, and then use them to give a proof of Theorem \ref{main thm intro}.
In section 4, we apply Theorem \ref{main thm intro} to the subcategories of projective/injective and Gorenstein projective/injective objects in an abelian category. In section 5, we extend the application of Theorem \ref{main thm intro} to the subcategories of $C$-projective and Gorenstein $C$-projective/injective modules, where $C$ denotes a fixed semidualizing module. In section 6, we further apply Theorem \ref{main thm intro} to a contravariantly finite resolving subcategory and an orthogonal subcategory.
\begin{conv}
Throughout this paper, we adopt the following conventions. All subcategories are assumed to be full and closed under isomorphisms. All complexes in an additive category are assumed to be cochain complexes unless stated otherwise. For an abelian category $\A$, we denote by $\proj\A$ (resp. $\inj\A$) the (full) subcategory of projective (resp. injective) objects in $\A$. In addition, we denote by $\D^{b}(\A)$ the bounded derived category of $\A$.
For a right noetherian ring  $R$, we denote by $\Mod R$ the category of right $R$-modules, and $\mod R$ the subcategory of $\Mod R$ consisting of finitely generated right $R$-modules. For simplicity, we denote $\proj(\Mod R)$, $\proj(\mod R)$, $\inj(\Mod R)$, and $\inj(\mod R)$ by $\Proj R$, $\proj R$, $\Inj R$, and $\inj R$ respectively.
We omit subscripts and superscripts unless confusion might occur.
\end{conv}

\section{Preliminaries}
In this section, we recall various fundamental definitions needed to state and prove the main theorem.
Let us begin with the notion of (co)resolving subcategories of an abelian category.
\begin{dfn}
Let $\A$ be an abelian category.
\begin{enumerate}[\rm(1)]
\item
A subcategory $\X$ of $\A$ is {\em resolving} if the following conditions hold.
\begin{enumerate}[\rm(a)]
\item
$\X$ contains $\proj \A$.
\item
$\X$ is closed under direct summands in $\A$, namely, if there exists an isomorphism $X \cong Y\oplus Z$ in $\A$ such that $X \in \X$, then $Y$ and $Z$ belong to $\X$.
\item
$\X$ is closed under extensions in $\A$, namely, if there exists an exact sequence $0\to X'' \to X \to X' \to 0$ such that $X''$, $X'$ $\in \X$, then $X$ belongs to $\X$.
\item
$\X$ is closed under kernels of epimorphisms in $\A$, namely, if there exists an exact sequence $0\to X'' \to X \to X' \to 0$ such that $X$, $X'$ $\in \X$, then $X''$ belongs to $\X$.
\end{enumerate}
\item
Dually, a subcategory $\X$ of $\A$ is {\em coresolving} if $\X$ contains $\inj \A$ and is closed under direct summands, extensions, and cokernels of monomorphisms in $\A$.
\item
Let $\X$ be a subcategory of $\A$.
Denote by $^{\perp}\X$ the subcategory of $\A$ consisting of objects $M$ satisfying that $\Ext_{\A}^{i}(M,\X) = 0$ for all $i>0$.
Dually, we denote by $\X^{\perp}$ the subcategory of $\A$ consisting of objects $M$ satisfying that $\Ext_{\A}^{i}(\X,M) = 0$ for all $i>0$.
\end{enumerate}
\end{dfn}
\begin{rmk}\label{rmk for resolving}
Let $\A$ be an abelian category, and $\X$ a subcategory of $\A$.
\begin{enumerate}[\rm(1)]
\item
Assume that $\X$ is closed under kernels of epimorphisms (resp. cokernels of monomorphisms), and $A,X$ are objects in $\A$ with $A\oplus X \in \X$, and $X\in\X$. Then by considering the exact sequence $0\to A\to A\oplus X \to X\to0$ (resp. $0\to X\to X\oplus A \to A\to0$), we have $A\in \X$.
\item
The subcategory $^{\perp}\X$ is a resolving subcategory of $\A$. Dually, the subcategory $\X^{\perp}$ is a coresolving subcategory of $\A$.
\end{enumerate}
\end{rmk}
We recall the concept of syzygies of an object in an abelian category for later use.

\begin{dfn}
Let $\A$ be an abelian category with enough projective objects, and $M$ an object in $\A$.
Take a $\proj \A$-resolution
$P=\left(\cdots \to P^{-i} \xrightarrow{d^{-i}_{P}} P^{-i+1} \xrightarrow{d^{-i+1}_{P}} \cdots \xrightarrow{d^{1}_{P}} P^{0}\to 0\right)$
of $M$.
For a positive integer $n$, we define the {\em $n$-th syzygy} of $M$, denoted by $\Omega^{n}M$, as the image of $d^{-n}_{P}$.
We set $\Omega^{0}M = M$. Note that the $n$-th syzygy of $M$ is uniquely determined up to projective summands by Schanuel's lemma.
\end{dfn}
Next, we recall various objects and truncations associated with a (cochain) complex.
\begin{dfn}
Let $\A$ be an abelian category, and $X=\left(\cdots \to X^{n-1}\xrightarrow{d_{X}^{n-1}} X^{n} \xrightarrow{d_{X}^{n}}X^{n+1}\xrightarrow{d_{X}^{n+1}}\cdots\right)$ a complex of objects in $\A$.
\begin{enumerate}[\rm(1)]
\item
For each integer $n$, we set 
$\cycle^{n}\!X = \ker d_{X}^{n}$,
$\cobound^{n}\!X = \coker d_{X}^{n-1}$,
$\bound^{n}\!X = \image d_{X}^{n-1}$, and
$\homol^{n}\!X = \cycle^{n}\!X/\bound^{n}\!X$.
In addition, we set
$\sup X = \sup\{i\in\ZZ \mid \homol^{i}\!X \neq 0\}$, and
$\inf X = \inf\{i\in\ZZ \mid \homol^{i}\!X \neq 0\}$.
\item 
Let $\X$ be a subcategory of $\A$ that contains a zero object.
We denote by $C(\X)$ the category of complexes of objects in $\X$. 
In addition, we denote by $C^{+}(\X)$, $C^{-}(\X)$, and $C^{b}(\X)$, the subcategory of $C(\X)$ whose complexes are left-bounded, right-bounded, and bounded respectively. When stating properties that hold in common, we denote them by $C^{*}(\X)$, where $*=+,-,b,\emptyset$.
\item
Let $n$ be an integer.
The {\em soft truncation} $\tau^{\geq n}X$ (resp. $\tau^{\leq n}X$) of $X$ at $n$ is defined as follows.
\begin{center}
$ X \to \tau^{\geq n}X = \left(0\to \cobound^{n}\!X\xrightarrow{}X^{n+1}\xrightarrow{d_{X}^{n+1}}X^{n+2}\xrightarrow{d_{X}^{n+2}}\cdots \right)$

$\left(\text{resp. } X \leftarrow\tau^{\leq n}X = \left(\cdots\to X^{n-2}\xrightarrow{d_{X}^{n-2}}X^{n-1}\xrightarrow{}\cycle^{n}\!X \to 0\right)\right)$.
\end{center}
Similarly, the {\em hard truncation} $\sigma^{\geq n}X$ (resp. $\sigma^{\leq n}X$) of $X$ at $n$ is defined as follows.
\begin{center}
$ X \leftarrow \sigma^{\geq n}X = \left(0\to X^{n}\xrightarrow{d_{X}^{n}}X^{n+1}\xrightarrow{d_{X}^{n+1}}X^{n+2}\xrightarrow{d_{X}^{n+2}}\cdots \right)$

$\left(\text{resp. } X \to\sigma^{\leq n}X = \left(\cdots\to X^{n-2}\xrightarrow{d_{X}^{n-2}}X^{n-1}\xrightarrow{d_{X}^{n-1}}X^{n} \to 0\right)\right)$.
\end{center}
\end{enumerate}
\end{dfn}
\begin{rmk}\label{rmk for truncation}
Let $X$ be a complex of objects in $\A$ with $s=\sup X$ $\in\mathbb{Z}$ (resp. $i=\inf X$ $\in\mathbb{Z}$). Then for all integers $n\geq s$ (resp. $n\leq i$), the natural morphism $X \leftarrow\tau^{\leq n}X$ (resp. $X \to \tau^{\geq n}X$) gives a quasi-isomorphism.
Similarly, for all integers $n\geq s$ (resp. $n\leq i$), the natural morphism $\cycle^{n}\!X\left[-n\right] \to \sigma^{\geq n}X$ (resp. $\sigma^{\leq n}X \to \cobound^{n}\!X\left[-n\right]$) gives a quasi-isomorphism.
\end{rmk}
We now define the (co)resolution dimension of an object in the derived category of an abelian category. 
For an abelian category $\A$, we denote by $\D^{+}(\A)$ (resp. $\D^{-}(\A)$) the subcategory of the derived category of $\A$ consisting of complexes with left-bounded (resp. right-bounded) cohomologies.
\begin{dfn}\label{def of resoldim}
Let $\A$ be an abelian category, and $\X$ a subcategory of $\A$.
For a nonzero object $M$ in $\D^{-}(\A)$, we define the $\X${\em -resolution dimension} of $M$, denoted by $\X\resoldim M$, as the infimum of integers $n$ such that there exists a complex 
\begin{center}
$X=\left(0\to X^{-n} \to X^{-n+1}\to\cdots\to X^{\sup M}\to0\right)$
\end{center}
with $X\cong M$ in $\D^{-}(\A)$, and $X^{i}\in\X$ for all $-n\leq i\leq\sup M$.
We set $\X\resoldim 0=-\infty$.
Dually, For a nonzero object $M$ in $\D^{+}(\A)$, we define the $\X${\em -coresolution dimension} of $M$, denoted by $\X\coresoldim M$, as the infimum of integers $n$ such that there exists a complex 
\begin{center}
$X=\left(0\to X^{\inf M} \to\cdots\to X^{n-1}\to X^{n}\to0\right)$
\end{center}
with $X\cong M$ in $\D^{+}(\A)$, and $X^{i}\in\X$ for all $\inf M\leq i\leq n$.
We set $\X\coresoldim 0=-\infty$.
\end{dfn}
\begin{rmk}\label{rmk for resoldim}
\begin{enumerate}[\rm(1)]
\item
For an object $M$ in $\D^{-}(\A)$, one has
\begin{center}
$\X\resoldim M = \inf\{-\inf\{n\in\mathbb{Z}\mid X^{n}\neq 0\}\mid X\in C^{-}(\X)\text{ and } X\cong M \text{ in } \D^{-}(\A)\}$.
\end{center}
Dually, for an object $M$ in $\D^{+}(\A)$, one has
\begin{center}
$\X\coresoldim M = \inf\{\sup\{n\in\mathbb{Z}\mid X^{n}\neq 0\}\mid X\in C^{+}(\X)\text{ and } X\cong M \text{ in } \D^{+}(\A)\}$.
\end{center}
\item
When $\X=\proj\A$, then the $\X$-resolution dimension of an object $M\in\D^{-}(\A)$ coincides with the projective dimension of $M$. Dually, when $\X=\inj\A$, then the $\X$-coresolution dimension of an object $M\in\D^{+}(\A)$ coincides with the injective dimension of $M$.
\item 
If $M$ is an object in $\A$, then by definition one has $\X\resoldim M \in \mathbb{Z}_{\geq 0}\cup \left\{\pm\infty\right\}$, and $\X\coresoldim M \in \mathbb{Z}_{\geq 0}\cup \left\{\pm\infty\right\}$.
In addition, each object in $\X$ has $\X$-resolutuion dimension at most zero.
\item
Assume that $\A$ has enough projective (resp. injective) objects. By considering a projective (resp. injective) resolution of each object in $\D^{b}(\A)$, we have the following equivalence of categories:
\begin{center}
$\D^{b}(\A)\cong \K^{-,b}(\proj\A)$\quad (resp. $\D^{b}(\A)\cong\K^{+,b}(\inj\A))$,
\end{center}
where $\K^{-,b}(\proj\A)$ (resp. $\K^{+,b}(\inj\A)$) is the right-bounded (resp. left-bounded) homotopy category of $\proj\A$ (resp. $\inj\A$) consisting of complexes with bounded cohomologies.
\end{enumerate}
\end{rmk}
We close this section by recalling the definitions of several basic concepts, including the notion of the level of an object in a triangulated category.
\begin{dfn}
Let $\T$ be a triangulated category.
\begin{enumerate}[\rm(1)]
\item
For an additive category $\mathcal{C}$ and a subcategory $\X$ of $\mathcal{C}$, we denote by $\add_{\mathcal{C}}\X$ the {\em additive closure} of $\X$ in $\mathcal{C}$, namely, the subcategory of $\mathcal{C}$ consisting of direct summands of finite direct sums of objects in $\X$.
\item
A subcategory $\mathcal{C}$ of $\T$ is {\em thick} if $\mathcal{C}$ is closed under shifts, cones, and direct summands in $\T$. In addition, for a subcategory $\X$ of $\T$, we denote by $\thick_{\T} \X$ the {\em thick closure} of $\X$ in $\T$. That is, the smallest thick subcategory of $\T$ containing $\X$.
\item
Let $\X$ be a subcategory of $\T$.
We denote by ${\langle\X\rangle}^{\T}$ the additive closure of the subcategory of $\T$ consisting of objects of the form $\Sigma^{n}X$, where $\Sigma$ is a shift functor on $\T$, $n \in \mathbb{Z}$, and $X\in\X$.
\item
For subcategories $\X$, $\Y$ of $\T$, we define $\X\ast\Y$ as the subcategory of $\T$ consisting of objects $E$ such that there exists an exact triangle $X\to E\to Y\to\Sigma X$ with $X\in\X$, and $Y\in\Y$.
\item
Let $\X$ be a subcategory of $\T$, and $r$ a nonnegative integer.
We inductively define ${\langle\X\rangle}_{r}^{\T}$ as follows.
\begin{equation}
{\langle \X \rangle}_{r}^{\T} = \nonumber
\begin{cases}
0 & \text{if $r=0$,} \\
{\langle \X \rangle}^{\T} & \text{if $r=1$,} \\
\left\langle{{\langle \X \rangle}_{r-1}^{\T}} \ast {\langle \X \rangle}^{\T}\right\rangle^{\T} & \text{if $r>1$.}
\end{cases}
\end{equation}
\item
For a subcategory $\X$ of $\T$ and an object $M$ in $\T$, we define the $\X$-{\em level} of $M$ in $\T$, denoted by $\level_{\T}^{\X}M$, as the infimum of nonnegative integers $n$ such that $M$ belongs to ${\langle\X\rangle}_{n}^{\T}$.
That is,
\begin{center}
$\level_{\T}^{\X}M=\inf\{n\geq0\mid M\in {\langle\X\rangle}_{n}^{\T} \}$.
\end{center}
\end{enumerate}
\end{dfn}
\begin{rmk}\label{rmk for thick}
Let $\T$ be a triangulated category, and $\X,\Y$ subcategories of $\T$.
\begin{enumerate}[\rm(1)]
\item 
For an object $M$ in $\T$, $M$ belongs to $\left\langle{{\langle \X \rangle}^{\T}} \ast {\langle \Y \rangle}^{\T}\right\rangle^{\T}$ if and only if there exists an exact triangle $X\to E\to Y\to \Sigma X$ in $\T$ such that $X\in{\langle \X \rangle}^{\T}$, $Y\in{\langle \Y \rangle}^{\T}$, and $M$ is a direct summand of $E$.
\item
There exists an ascending chain of subcategories of $\T$ as follows:
\begin{center}
$0={\langle \X \rangle}_{0}^{\T} \subseteq {\langle \X \rangle}_{1}^{\T} \subseteq \cdots\subseteq {\langle \X \rangle}_{n}^{\T}\subseteq\cdots\subseteq \displaystyle\bigcup_{i\geq0}{\langle \X \rangle}_{i}^{\T} = \thick_{\T}\X$.
\end{center}
\item
Let $\mathcal{S}$ be a thick subcategory of $\T$ that contains $\X$. Then one has $\level_{\T}^{\X}M=\level_{\mathcal{S}}^{\X}M$ for all objects $M\in\T$.
\end{enumerate}
\end{rmk}
\section{Proof of main theorem}
This section is devoted to proving our main result, Theorem \ref{main thm llevel}.

To establish Lemma \ref{CFH 2006 3.1.2}, we provide the following lemma, which is proved by using the techniques of \cite[Theorem 3.1]{Christensen Frankild Holm}.
\begin{lem}\label{CFH 2006 3.1.1}
Let $\A$ be an abelian category, and $\X$ a subcategory of $\A$ that contains $\proj\A$ and is closed under extensions and kernels of epimorphisms in $\A$.
Let $f:A\to B$ be a quasi-isomorphism of complexes in $C^{-}(\X)$
Then for all integers $n$,
$\cobound^{n}\!A$ is in $\X$
if and only if $\cobound^{n}\!B$ is in $\X$.
\end{lem}
\begin{proof}
Fix an integer $n$.
Since $f$ is a quasi-isomorphism, the induced morphism 
$f' : \tau^{\geq n}A \to \tau^{\geq n}B$ is also a quasi-isomorphism.
Hence the mapping cone of $f'$ is acyclic and of the following form.
\begin{center}
$\cone(f')=\left(
0\to\cobound^{n}\!A\to \cobound^{n}\!B \oplus A^{n+1} \to B^{n+1}\oplus A^{n+2}\to\cdots\to 0
\right)$.
\end{center}
As $\X$ is closed under finite direct sums and kernels of epimorphisms in $\A$, there exists an exact sequence
\begin{center}
$0\to\cobound^{n}\!A\to \cobound^{n}\!B \oplus A^{n+1} \to X \to 0$
\end{center}
such that the object $X$ belongs to $\X$.

Assume that $\cobound^{n}\!A$ belongs to $\X$.
Then $\cobound^{n}\!B \oplus A^{n+1}$ belongs to $\X$ as $\X$ is closed under extensions.
Hence $\cobound^{n}\!B$ is the kernel of the epimorphism $\cobound^{n}\!B \oplus A^{n+1} \to A^{n+1}$ in $\X$, and thus $\cobound^{n}\!B$ belongs to $\X$.
The converse is an immediate consequence of the above short exact sequence.
\end{proof}
The following lemma plays an important role in determining the finiteness of the resolution dimension of complexes. In the case of Gorenstein projective dimension, it is stated in \cite[Theorem 3.1]{Christensen Frankild Holm}.
\begin{lem}\label{CFH 2006 3.1.2}
Let $\A$ be an abelian category with enough projective objects, and $\X$ a subcategory of $\A$ that contains $\proj\A$ and is closed under extensions and kernels of epimorphisms in $\A$.
Let $M$ be an object in $\D^{b}(\A)$, and $n$ an integer.
Then the following are equivalent.
\begin{enumerate}[\rm(1)]
\item
The inequality $\X\resoldim M \leq n$ holds.
\item
The inequality $-\inf M \leq n$ holds, and for all complexes $X$ in $C^{-}(\X)$ with $X \cong M$ in $\D^{b}(\A)$,
the object $\cobound^{-n}\!X$ is in $\X$.
\end{enumerate}
\end{lem}
\begin{proof}
(2)$\Rightarrow$(1):
Take a $\proj\A$-resolution $P\xrightarrow{\cong}M$ with $P\in\K^{-,b}(\proj\A)$, and $P^{i}=0$ for all $i>\sup M$.
Then by assumption, $\cobound^{-n}\!P$ belongs to $\X$, and $-n\leq\inf M=\inf P$ holds.
Hence we have the isomorphisms
\begin{center}
$M\cong P\cong\tau^{\geq -n}P=\left(0\to\cobound^{-n}\!P\to P^{-n+1}\to\cdots\to P^{\sup M}\to 0\right)$
\end{center}
in $\D^{b}(\A)$. 
This implies that $M$ has $\X$-resolution dimension at most $n$.

(1)$\Rightarrow$(2): 
By assumption, $M$ is isomorphic to a complex $A=\left(0\to A^{-n}\to\cdots\to A^{\sup M}\to 0\right)$ in $\D^{b}(\A)$ with $A^{i} \in \X$ for all $i$.
Then one has $\cobound^{-n}\!A=A^{-n}\in\X$, and $-n\leq\inf A=\inf M$.
Assume that $M$ is isomorphic to a complex $X$ in $\D^{b}(\A)$ with $X\in C^{-}(\X)$.
Take a $\proj\A$-resolution $P\xrightarrow{\cong}M$ with $P\in\K^{-,b}(\proj\A)$, and $P^{i}=0$ for all $i>\sup M$.
Since $P$ is a right-bounded complex with projective objects in $\A$, the natural homomorphism $\Hom_{\K(\A)}(P,Y)\to\Hom_{\D(\A)}(P,Y)$ is an isomorphism for each object $Y$ in $\D(\A)$.
Thus there exist quasi-isomorphisms $P\to A$, and $P\to X$.
Hence by Lemma \ref{CFH 2006 3.1.1}, $\cobound^{-n}\!P$ belongs to $\X$, and so does $\cobound^{-n}\!X$.
\end{proof}
In Lemma \ref{CFH 2006 3.1.2}, restricting to the case where $M$ is an object in $\A$, we obtain the following result, which characterizes the resolution dimension using syzygies of $M$ (cf. \cite[Lemma 4.3]{stcm}).
\begin{cor}\label{CFH 2006 3.1.2 cor}
Let $\A$ be an abelian category with enough projective objects, and $\X$ a subcategory of $\A$ that contains $\proj\A$ and is closed under extensions and kernels of epimorphisms in $\A$.
Let $A$ be an object in $\A$, and $n$ a nonnegative integer.
Then the following are equivalent.
\begin{enumerate}[\rm(1)]
\item
The inequality $\X\resoldim A \leq n$ holds.
\item
The object $\Omega^{n}A$ is in $\X$.
\end{enumerate}
\end{cor}
\begin{proof}
We may assume that $n>0$.
Take a $\proj\A$-resolution $P$ of $A$.
Then we have
$\Omega^{n}A = \bound^{-n+1}\!P \cong P^{-n}/\cycle^{-n}\!P = P^{-n}/\bound^{-n}\!P = \cobound^{-n}\!P$ as $P$ is exact at the degree $n$.
Hence the desired equivalence follows directly from Lemma \ref{CFH 2006 3.1.2}.
\end{proof}
The following lemma is stated in \cite[Lemma 3.1]{Awadalla Marley} for the case where the objects are modules.
However it holds in any abelian category with the same proof.
\begin{lem}[{\cite[Lemma 3.1]{Awadalla Marley}}]\label{AW 3.1}
Let $\A$ be an abelian category.
Consider the following commutative diagram in $\A$ with exact rows.
\begin{equation*}
\xymatrix{
0 \ar[r] & A \ar[r]^{f} \ar[d]^{\alpha} & B \ar[r]\ar[d] & C \ar[r]\ar[d]^{\gamma} & 0 \\
0 \ar[r] & D \ar[r] & E \ar[r] & F \ar[r] & 0.
}
\end{equation*}
Assume that there exists a morphism
$\phi : B \to D$ in $\A$ such that $\alpha = \phi f$, and the induced homomorphism 
$\Ext_{\A}^{1}(\gamma, D) : \Ext_{\A}^{1}(F, D) \to \Ext_{\A}^{1}(C, D)$ is injective.
Then the bottom row splits.
\end{lem}
The following lemma is an immediate consequence of Remark \ref{rmk for resolving}, and a special case of it is stated in \cite[Lemma 2.1]{Christensen Frankild Holm}.
\begin{lem}\label{CFH 2.1}
Let $\A$ be an abelian category, and $\X$ a subcategory of $\A$.
Then for each object in $\A$ with finite $\X^{\perp}$-resolution dimension belongs to $\X^{\perp}$.
\end{lem}
The following lemma is a generalization of \cite[Lemma 3.2]{Awadalla Marley}. However, we omit the proof as the same argument is valid.
\begin{lem}\label{AM 3.2}
Let $\A$ be an abelian category, and 
$f : P \to Q$ a quasi-isomorphism of complexes in $C^{-}(\X)$.
Then for all objects $M$ in $\X^{\perp}$, positive integers $i$, and integers $v$, the induced homomorphisms
\begin{center}
$\Ext_{\A}^{i}(\cobound^{v}\!Q, M) \to \Ext_{\A}^{i}(\cobound^{v}\!P, M), \quad
\Ext_{\A}^{i}(\bound^{v}\!Q, M) \to \Ext_{\A}^{i}(\bound^{v}\!P, M)$
\end{center}
are isomorphisms.
\end{lem}
The finiteness of the projective dimension of a complex can be expressed in terms of the vanishing of Ext modules (cf. \cite[A.5.3]{Christensen t}). In addition, the finiteness of the Gorenstein projective dimension of a complex is characterized by the property of being a reflexive complex (cf. \cite[Theorem 2.7]{Yassemi}). These results imply the thickness of the subcategory of complexes with finite (Gorenstein) projective dimension.
More generally, the following proposition ensures that the category of complexes with finite resolution dimension by a resolving subcategory forms a thick subcategory of the bounded derived category.
\begin{prop}\label{thickness}
Let $\A$ be an abelian category with enough projective objects, and $\X$ a subcategory of $\A$ that contains $\proj\A$ and is closed under extensions and kernels of epimorphisms in $\A$.
We denote by $\mathcal{C}$ the subcategory of $\D^{b}(\A)$ consisting of complexes of finite $\X$-resolutuion dimension.
Then the following hold.
\begin{enumerate}[\rm(1)]
\item
The subcategory $\mathcal{C}$ is a triangulated subcategory of $\D^{b}(\A)$.
\item
Furthermore, assume that $\X$ is closed under direct summands (i.e., $\X$ is resolving), then $\mathcal{C}$ is a thick subcategory of $\D^{b}(\A)$.
\end{enumerate}
\end{prop}
\begin{proof}
(1) It is obvious that $\mathcal{C}$ is closed under shifts. We shall prove that $\mathcal{C}$ is closed under mapping cones.
Let $A\to B\to C\rightsquigarrow$ be an exact triangle in $\D^{b}(\A)$ with $A$, $B \in \mathcal{C}$.
Take $\proj\A$-resolutions $P\xrightarrow{\simeq}A$, and $Q\xrightarrow{\simeq}B$ with $P$, $Q\in\K^{-,b}(\proj\A)$. 
Then  we have the following diagram in $\D^{b}(\A)$ with exact rows.
\begin{equation*}
\xymatrix{
P \ar[r]^{\Phi}\ar[d]^{\simeq}\ar@{}[dr]|\circlearrowright & Q \ar[r]\ar[d]^{\simeq}\ar@{}[dr]|\circlearrowright & \cone\Phi \ar@{-->}[d]^{\simeq} \rightsquigarrow \\
A \ar[r] & B \ar[r] & \quad C \rightsquigarrow,
}
\end{equation*}
where $\Phi$ is the composition morphism of $P\to A\to B\to Q$ in $\D^{b}(\A)$.
As $P$ is an object in $\K^{-,b}(\proj\A)$, the natural homomorphism $\Hom_{\K(\A)}(P,Q)\to \Hom_{\D(\A)}(P,Q)$ is an isomorphism. Hence, there exists a morphism $\varphi : P\to Q$ of complexes such that $\cone(\varphi)$ is isomorphic to $\cone\Phi$ in $\D^{b}(\A)$.
Thus, it is enough to show that the $\X$-resolutuion dimension of $\cone(\varphi)$ is finite.
Let $n=\sup\{\X\resoldim A,\ \X\resoldim B \} < \infty$.
We may assume that $n$ is an integer.
Then by Lemma \ref{CFH 2006 3.1.2}, $\cobound^{-m}\!P$ and $\cobound^{-m}\!Q$ belong to $\X$ for all $m\geq n$.
Note that the inequality $-n\leq \inf\{\inf P,\ \inf Q\}$ holds.
By applying hard truncations to the exact sequence of complexes
\begin{center}
$0\to Q\to\cone(\varphi)\to P\left[1\right]\to 0$
\end{center}
at degree $-n-1$, we have the following exact sequence of complexes.
\begin{equation*}
\xymatrix{
0 \ar[r] & \sigma^{\leq -n-1}Q \ar[r] &\sigma^{\leq -n-1}\cone(\varphi) \ar[r] & \sigma^{\leq -n-1}P\left[1\right] \ar[r] & 0.
}
\end{equation*}
Namely, there exists a diagram with exact rows as follows.
\begin{equation*}
\xymatrix{
& {\rotatebox{270}{$\cdots$}} \ar[d]^{d_{Q}^{-n-3}} & {\rotatebox{270}{$\cdots$}} \ar[d]^{d_{\cone(\varphi)}^{-n-3}} & {\rotatebox{270}{$\cdots$}} \ar[d]^{-d_{P}^{-n-2}} & \\
0 \ar[r] & Q^{-n-2} \ar[d]^{d_{Q}^{-n-2}}\ar[r] & Q^{-n-2}\oplus P^{-n-1} \ar[d]^{d_{\cone(\varphi)}^{-n-2}}\ar[r] & P^{-n-1} \ar[d]^{-d_{P}^{-n-1}}\ar[r] & 0 \\
0 \ar[r] & Q^{-n-1} \ar[d]\ar[r] & Q^{-n-1}\oplus P^{-n} \ar[d]\ar[r] & P^{-n} \ar[d]\ar[r] & 0 \\
& 0 & 0 & 0 &
}
\end{equation*}
Hence, by taking the long exact sequence of cohomologies, we have the following exact sequence in $\A$.
\begin{center}
$0\to \cobound^{-n-1}Q\!\to \cobound^{-n-1}\cone(\varphi)\to \cobound^{-n}P\to 0$.
\end{center}
Since $\X$ is closed under extensions, $\cobound^{-n-1}\cone(\varphi)$ belongs to $\X$. Therefore the $\X$-resolution dimension of $\cobound^{-n-1}\cone(\varphi)$ is less than or equal to $n+1$ by Lemma \ref{CFH 2006 3.1.2}, and thus it is finite.
  
(2) By virtue of (1), it suffices to show that $\mathcal{C}$ is closed under direct summands in $\D^{b}(\A)$. Let $M,M_{1}$, and $M_{2}$ be objects in $\D^{b}(\A)$ with $M\in \mathcal{C}$ and $M\cong M_{1}\oplus M_{2}$ in $\D^{b}(\A)$. Then we may assume that $n=\X\resoldim M$ is an integer. Firstly, we consider the case where $M,M_{1}$, and $M_{2}$ are objects in $\A$. Since the natural functor $\A \to \D^{b}(\A)$ is fully faithful, we have an isomorphism $M\cong M_{1}\oplus M_{2}$ in $\A$. Hence one has $\Omega^{i}M\cong \Omega^{i}M_{1}\oplus \Omega^{i}M_{2}$ in $\A$ for all $i\geq 0$. As $n=\X\resoldim M \in \mathbb{Z}_{\geq 0}$, the object $\Omega^{n}M$ belongs to $\X$ by Corollary \ref{CFH 2006 3.1.2 cor}. Since $\X$ is closed under direct summands, $\Omega^{n}M_{1}$ and $\Omega^{n}M_{2}$ belong to $\X$. Hence the $\X$-resolution dimensions of $M_{1}$ and $M_{2}$ are less than or equal to $n$ by Corollary \ref{CFH 2006 3.1.2 cor}. Therefore $M_{1}$ and $M_{2}$ belong to $\mathcal{C}$.
Next, we consider the general case. Fix an integer $i\in \{1,2\}$. Take a $\proj\A$-resolution $P_{i}\xrightarrow{\simeq}M_{i}$ of $M_{i}$. As the inequalities $-n\leq \inf M\leq \inf M_{i}$ hold, there exists an exact triangle
\begin{equation}\label{thickness ex.tri}
F_{i}\to M_{i}\to A_{i}\left[n\right]\rightsquigarrow
\end{equation}
in $\D^{b}(\A)$ such that $F_{i}$ has finite $\proj\A$-resolution dimension and $A_{i}$ belongs to $\A$ by \cite[Lemma 2.4 (1)]{Dao Takahashi 2015}. By taking a direct sum of two exact triangles arising from the case where $i=1,2$ in (\ref{thickness ex.tri}), we have the following exact triangle in $\D^{b}(\A)$.
\begin{center}
$F_{1}\oplus F_{2}\to M_{1}\oplus M_{2}\to A_{1}\oplus A_{2}\left[n\right]\rightsquigarrow$.
\end{center}
Since $F_{1}\oplus F_{2}$ and $M_{1}\oplus M_{2} \cong M$ have finite $\X$-resolution dimension, $A_{1}\oplus A_{2}$ also has finite $\X$-resolution dimension by (1). Hence, from the first case we considered, $A_{1}$ and $A_{2}$ belong to $\mathcal{C}$. Thus, each $M_{i}$ belongs to $\mathcal{C}$ for $i=1,2$ by (\ref{thickness ex.tri}) and (1).
\end{proof}
We introduce the following condition $(\mathbf{A})$, which is a requirement on the resolving subcategory necessary to establish the proof for obtaining the lower bounds.
\begin{dfn}\label{condition A}
Let $\A$ be an abelian category, and $\X$ a subcategory of $\A$.
We say that $\X$ satisfies the condition $(\mathbf{A})$ if for all objects $X$ in $\X$, there exists an exact sequence
$0\to X \to Y \to X' \to 0$ 
such that $Y\in \X \cap \X^{\perp}$ and $X' \in \X$.
Dually, we say that $\X$ satisfies the condition $(\mathbf{A}^{*})$ if for all objects $X$ in $\X$, there exists an exact sequence
$0\to X' \to Y \to X \to 0$ 
such that $Y\in \X \cap {}^{\perp}\X$ and $X' \in \X$.
\end{dfn}
The following lemma plays an important role in proving our Theorem \ref{main thm llevel} and is motivated by \cite[Theorem 3.1 and Lemma 3.4]{Christensen Iyengar}. The condition $(\mathbf{A})$ required for a resolving subcategory is used to construct an approximation of a complex with finite resolution dimension.
\begin{lem}\label{CI 3.1}
Let $\A$ be an abelian category.
\begin{enumerate}[\rm(1)]
\item
Let $X$ be a complex of $\A$ and  $n$ an integer.
Assume that there exists a monomorphism
$\iota : X^{n} \to Y^{n}$ in $\A$.
Then we have the following exact sequence of complexes.
\begin{equation*}
\xymatrix{
0 \ar[r] & X \ar@{}[d]|{\rotatebox{270}{$\quad=(\cdots$}} \ar[r] & Y \ar@{}[d]|{\rotatebox{270}{$\quad=(\cdots$}} \ar[r] & C \ar[r]  & 0 \\
& \ar[d] & \ar[d] & & \\
0 \ar[r] & X^{n-1}\ar@{=}[r] \ar[d]^{d_{X}^{n-1}} & X^{n-1} \ar[d] \ar[r] & 0 \ar[d] & \\
0 \ar[r] & X^{n} \ar[r]^{\iota} \ar[d]^{d_{X}^{n}} & Y^{n} \ar[r] \ar[d] & \coker \iota \ar[r] \ar@{=}[d] & 0 \\
0 \ar[r] & X^{n+1} \ar[r] \ar[d]^{d_{X}^{n+1}} & Y^{n+1} \ar[r] \ar[d]^{f} & \coker \iota \ar[r] \ar[d] & 0 \\
0 \ar[r] & X^{n+2}\ar@{=}[r] \ar[d] & X^{n+2} \ar[d] \ar[r] & 0 & \\
& \rotatebox{90}{$(\cdots$}& \rotatebox{90}{$(\cdots$} && \\
}
\end{equation*}
\item
Assume that $\A$ has enough projective objects, and
let $\X$ be a subcategory of $\A$ that contains $\proj\A$ and is closed under extensions and kernels of epimorphisms in $\A$, and satisfies the condition $(\mathbf{A})$.
Then for all objects $M$ in $\D^{b}(\A)$ 
with finite $\X$-resolutuon dimension 
and integers $n$ with $-\sup M \leq n \leq \X\resoldim M$,
there exists a complex $C(n)$ in $\D^{b}(\A)$ such that 
$C(n)$ is isomorphic to $M$ in $\D^{b}(\A)$ 
and $C(n)$ is of the following form:
\begin{center}
$C(n)=\left(0 \to Q^{-g}\to \cdots \to Q^{-n-1} \to X^{-n}\to P^{-n+1} \to \cdots \to P^{s}\to 0\right)$,
\end{center}
where 
$g=\X\resoldim M$, $s = \sup M$, 
$Q^{i}\in \X \cap \X^{\perp}$ for $i = -g,\ldots, -n-1$, 
$X^{-n} \in \X$, and
$P^{i} \in \proj\A$ for $i = -n+1,\ldots, s$.
\end{enumerate}
\end{lem}
\begin{proof}
(1)
The object $Y^{n+1}$ is constructed by the pushout of $\iota$ and $d_{X}^{n}$ in $\A$, and the morphism $f : Y^{n+1}\to X^{n+2}$ is induced by the universal property of the pushout diagram. From the commutativity of the diagram and the uniqueness possessed by the induced morphism $f$, we obtain the complex $Y$ as stated in the claim.

(2) We shall construct the complex $C(n)$ by the descending induction on $n$.
When $n=g$, Take a $\proj\A$-resolution $P$ of $M$.
As the inequality $-g\leq\inf P$ holds, one has
\begin{center}
$M\cong P\cong\tau^{\geq -g}P =\left(0\to X^{-g}\to P^{-g+1}\to\cdots\to P^{s}\to 0\right)$,
\end{center}
where $X^{-g}=\cobound^{-g}\!P\in\X$ by Lemma \ref{CFH 2006 3.1.2}.
Thus, setting $C(g)=\left(0\to X^{-g}\to P^{-g+1}\to\cdots\to P^{s}\to 0\right)$ suffices.
Assume that $n<g$, and we construct $C(n)$ from
\begin{center}
$M\cong C(n+1)=\left(0 \to Q^{-g}\to \cdots\to Q^{-n-2}\to X^{-n-1}\to P^{-n} \to \cdots \to P^{s}\to 0\right)$,
\end{center}
where 
$Q^{i}\in \X \cap \X^{\perp}$ for $i = -g,\ldots, -n-2$, 
$X^{-n-1} \in \X$, and
$P^{i} \in \proj\A$ for $i = -n,\ldots, s$.
Since $X^{-n-1}$ belongs to $\X$, there exists an exact sequence
\begin{center}
$0\to X^{-n-1}\to Q^{-n-1}\to X'\to 0$,
\end{center}
where $Q^{-n-1}\in \X\cap\X^{\perp}$, and $X'\in \X$ by the condition $(\mathbf{A})$.
Hence by $(1)$, we have the following quasi-isomorphism of complexes.
\begin{equation*}
\xymatrix@C=15pt@R=20pt{
&&&&& 0\ar[d] & 0\ar[d] &&&& \\
C(n+1) \ar@{=}[r]\ar[d]^{\cong} &(0\ar[r] & Q^{-g}\ar[r]\ar@{=}[d] & \cdots\ar[r] & Q^{-n-2}\ar[r]\ar@{=}[d] & X^{-n-1}\ar[r]\ar[d] & P^{-n}\ar[r]\ar[d] & P^{-n+1}\ar[r]\ar@{=}[d] & \cdots\ar[r] & P^{s}\ar[r]\ar@{=}[d] & 0) \\
C(n) \ar@{=}[r] &(0\ar[r] & Q^{-g}\ar[r] & \cdots\ar[r] & Q^{-n-2}\ar[r] & Q^{-n-1}\ar[r]\ar[d] & X^{-n}\ar[r]\ar[d] & P^{-n+1}\ar[r] & \cdots\ar[r] & P^{s}\ar[r] & 0). \\
&&&&& X'\ar@{=}[r]\ar[d] & X'\ar[d] &&&& \\
&&&&& 0 & 0 &&&& 
}
\end{equation*}
As $\X$ is closed under extensions, $X^{-n}$ belongs to $\X$.
Therefore, the above complex $C(n)$ is the desired one.
\end{proof}
Next, we recall the definition of a (co)ghost map in a triangulated category.
\begin{dfn}
Let $\T$ be a triangulated category with a shift functor $\Sigma$, and $\mathcal{C}$ a subcategory of $\T$.
A morphism $f:M\to N$ in $\T$ is {\em $\mathcal{C}$-ghost} (resp. {\em $\mathcal{C}$-coghost}) if the induced morphisms $\Hom_{\T}(\Sigma^{n}C, f): \Hom_{\T}(\Sigma^{n}C, M)\to \Hom_{\T}(\Sigma^{n}C, N)$ (resp. $\Hom_{\T}(f,\Sigma^{n}C)$: $\Hom_{\T}(N, \Sigma^{n}C)\to \Hom_{\T}(M, \Sigma^{n}C))$ are zero for all objects $C$ in $\mathcal{C}$ and integers $n$. 
\end{dfn}
\begin{rmk}\label{rmk for ghost}
Let $\A$ be an abelian category. For the triangulated category $\T=\D(\A)$ and a morphism $f : M\to N$ in $\D(\A)$, $f$ is $\mathcal{C}$-ghost (resp. $\mathcal{C}$-coghost) if and only if the induced morphisms of cohomologies $\Ext_{\A}^{n}(C,f) : \Ext_{\A}^{n}(C,M)\to\Ext_{\A}^{n}(C,N)$ (resp. $\Ext_{\A}^{n}(f,C) : \Ext_{\A}^{n}(N,C)\to\Ext_{\A}^{n}(M,C)$) are zero for all $C$ in $\mathcal{C}$, and integers $n$.
\end{rmk}
Now we have reached the main result of this section, which is established by combining the lemmas proved above. Its techniques of proof are based on \cite[Theorem 3.3]{Awadalla Marley}.
\begin{thm}\label{main thm llevel}
Let $\A$ be an abelian category with enough projective objects, and $\X$ a resolving  subcategory of $\A$ satisfying the condition $(\mathbf{A})$.
Then for all nonzero objects $M$ in $\D^{b}(\A)$, the following inequality holds.
\begin{center}
$\level_{\D^{b}(\A)}^{\X}M 
\geq \X \resoldim M + \inf M +1$.
\end{center}
\end{thm}
\begin{proof}
We may assume that $u=\level_{\D^{b}(\A)}^{\X}M$ is a positive integer. Then $M$ belongs to ${\langle \X \rangle}_{u}^{\D^{b}(\A)}$. Combining Remarks \ref{rmk for resoldim}, \ref{rmk for thick}, and Proposition \ref{thickness}, each object in ${\langle \X \rangle}_{u}^{\D^{b}(\A)}$ has finite $\X$ resolution dimension. Thus $g=\X\resoldim M$ is finite.

We set $i=\inf M$ and $s=\sup M$. Then one has $-\infty<-g\leq i\leq s<\infty$. We may assume that $-g<i$. By Lemma \ref{CI 3.1}, $M$ is isomorphic to the following complex in $\D^{b}(\A)$.
\begin{center}
$C(-i)=\left(0\to X^{-g}\to\cdots\to X^{i}\to\cdots\to X^{s}\to 0 \right)=X$,
\end{center}
where $X^{j}\in\X\cap\X^{\perp}$ for $j=-g,\ldots,i-1$, $X^{i}\in\X$, and $X^{j}\in\proj\A$ for $j=i+1,\ldots,s$.
Note that $u=\level_{\D^{b}(\A)}^{\X}X$, $g=\X \resoldim X$, $i=\inf X$, and $s=\sup X$. 

For each integer $j$, we define $\varphi_{j}$ as the natural morphism of hard truncations $\sigma^{\leq j}X\to\sigma^{\leq j-1}\left(\sigma^{\leq j}X\right) =\sigma^{\leq j-1}X$.
We shall show that each $\varphi_{j}$ is $\X$-ghost in $\D(\A)$ for all integers $j\leq i$. By virtue of Remarks \ref{rmk for truncation} and \ref{rmk for ghost}, it is enough to show that the induced morphisms $\Ext_{\A}^{k}(A,\cobound^{j}\!X)\to\Ext_{\A}^{k+1}(A,\cobound^{j-1}\!X)$ are zero for all $A\in\X$ and $k\in\mathbb{Z}$. When $k<0$, the domain of the morphism vanishes since $A$ and $\cobound^{j}\!X$ are objects in $\A$. We consider the case where $k\geq0$. Since $j-1$ is less than $i$, $\cobound^{j-1}\!X \cong \left(\sigma^{\leq j-1}X\right)\left[j-1\right]$ has finite $\X^{\perp}$-resolution dimension. Hence by Lemma \ref{CFH 2.1}, $\cobound^{j-1}\!X$ belongs to $\X^{\perp}$. This implies that $\Ext_{\A}^{k+1}(A,\cobound^{j-1}\!X)$ vanishes for all $A\in\X$ and $k\geq0$.

We focus on the morphisms of hard truncations:
\begin{equation*}
X=\sigma^{\leq s}X \xrightarrow{\varphi_{s}}\cdots\xrightarrow{\varphi_{i+1}}\sigma^{\leq i}X\xrightarrow{\varphi_{i}}\sigma^{\leq i-1}X\to\cdots\xrightarrow{\varphi_{-g+1}}\sigma^{\leq -g}X=X^{-g}\left[g\right].
\end{equation*}
We set $\rho=\varphi_{i+1}\cdots\varphi_{s}$,\ $\varphi'_{i}=\varphi_{i}\rho$, and $\psi=\varphi_{-g+1}\cdots\varphi_{i-1}\varphi'_{i}$.
Since $\varphi_{i}$ is $\X$-ghost, so is $\varphi'_{i}$. Hence $\psi$ is a composition of $g+i$ $\X$-ghost maps in $\D^{b}(\A)$. By virtue of the Ghost lemma (cf. \cite[Theorem 2.11]{Awadalla Marley} or \cite[Lemma 4.11]{Rouquier}), it is enough to show that $\psi$ is nonzero.
Assume that $\psi$ is zero in $\D^{b}(\A)$. Take a $\proj\A$-resolution $\sigma : P\to X$ with $P\in\K^{-,b}(\proj\A)$ and $P^{l}=0$ for all $l>\sup X$. Then the composition $\psi\sigma$ is also zero in $\D^{b}(\A)$. Since $P$ is a right-bounded complex with projective objects in $\A$, the morphism $\psi\sigma : P\to\sigma^{\leq -g}X$ is null-homotopic. Hence there exists a morphism $\tau : P^{-g+1}\to X^{-g}$ such that $\tau d_{P}^{-g}=\sigma^{-g}$. Thus we have the following diagram of morphisms.
\begin{equation*}
\xymatrix@=20pt{
&P^{-g} \ar[dd]_{\sigma^{-g}}\ar[rr]^{d_{P}^{-g}}\ar@{->>}[rd] && P^{^g+1}\ar[dd]^{\sigma^{-g+1}}\ar@/^16pt/[lldd]^{\tau} \\
&& \bound^{-g+1}\!P\ar@{^{(}-_>}[ru]^{\overline{d_{P}^{-g}}}\ar@/_8pt/[ld]^{\eta}   & \\
0 \ar[r] & X^{-g} \ar[rr]_{d_{X}^{-g}}&& X^{-g+1}
}
\end{equation*}
Since the inequality $-g<i$ holds, $d_{X}^{-g}$ is a monomorphism and $\eta$ is induced by the universal property of $X^{-g}$. Through a diagram chase in the above diagram, we can see that $\tau\overline{d_{P}^{-g}}=\eta$. Hence we obtain the following diagram with exact rows and commutative squares.
\begin{equation*}
\xymatrix{
0 \ar[r] & \bound^{-g+1}\!P \ar[r]^{\overline{d_{P}^{-g}}}\ar[d]^{\eta} & P^{-g+1}\ar[r]\ar[d]^{\sigma^{-g+1}}\ar[ld]^{\tau}\ar@{}@<-2.0ex>[dl]|{\circlearrowright} & \cobound^{-g+1}\!P \ar[r]\ar[d] & 0 \\
0 \ar[r] & X^{-g} \ar[r]_{d_{X}^{-g}} & X^{-g+1} \ar[r] & \cobound^{-g+1}\!X \ar[r] & 0.
}
\end{equation*}
Since $\sigma : P\to X$ is a quasi-isomorphism in $C^{-}(\X)$ and $X^{-g}$ belongs to $\X^{\perp}$, the induced homomorphism $\Ext_{\A}^{1}(\cobound^{-g+1}\!X,X^{-g})\to\Ext_{\A}^{1}(\cobound^{-g+1}\!P,X^{-g})$ is an isomorphism by Lemma \ref{AM 3.2}. By virtue of Lemma \ref{AW 3.1}, the bottom row in the above diagram splits. Hence $\cobound^{-g+1}\!X$ is a direct summand of $X^{-g+1}$, and it belongs to $\X$. Thus $g=\X\resoldim X$ is less than or equal to $g-1$ by Lemma \ref{CFH 2006 3.1.2}. This is a contradiction, and the assertion holds. 
\end{proof}
We close this section by providing the dual statement of the previous theorem. 
\begin{thm}
Let $\A$ be an abelian category with enough injective objects, and $\X$ a coresolving  subcategory of $\A$ satisfying the condition $(\mathbf{A}^{*})$.
Then for all nonzero objects $M$ in $\D^{b}(\A)$, the following inequality holds.
\begin{center}
$\level_{\D^{b}(\A)}^{\X}M 
\geq \X \coresoldim M - \sup M +1$.
\end{center}
\end{thm}
\begin{proof}
Applying Theorem \ref{main thm llevel} to the opposite category of $\A$, we obtain the desired inequality.
\end{proof}
\section{(Gorenstein) projective/injective levels}

In this section, we apply Theorem \ref{main thm llevel} to establish lower bounds for the (Gorenstein) projective and injective levels of complexes. We start by recalling the Gorenstein projective and injective dimension for a complex in an abelian category.
\begin{dfn}\label{def of Gproj}
Let $\A$ be an abelian category.
\begin{enumerate}[\rm(1)]
\item
Let $M$ be an object in $\A$.
A {\em complete resolution} of $M$ in $\A$ is an exact sequence
\begin{center}
$P=\left(\cdots\xrightarrow{d_{P}^{-2}}P^{-1}\xrightarrow{d_{P}^{-1}}P^{0}\xrightarrow{d_{P}^{0}} P^{1}\xrightarrow{d_{P}^{1}}\cdots \right)$
\end{center}
of projective objects in $\A$ such that the complex $\Hom_{\A}(P,Q)$ is exact for all $Q$ in $\proj\A$, and $M$ is isomorphic to $\cobound^{0}\!P$. Dually, a {\em complete coresolution} of $M$ in $\A$ is an exact sequence
\begin{center}
$I=\left(\cdots\xrightarrow{d_{I}^{-2}}I^{-1}\xrightarrow{d_{I}^{-1}}I^{0}\xrightarrow{d_{I}^{0}} I^{1}\xrightarrow{d_{I}^{1}}\cdots \right)$
\end{center}
of injective objects in $\A$ such that the complex $\Hom_{\A}(J,I)$ is exact for all $J$ in $\inj\A$, and $M$ is isomorphic to $\cycle^{0}\!I$.
\item
An object $M$ in $\A$ is {\em Gorenstein projective} (resp. {\em Gorenstein injective}) if $M$ has a complete resolution (resp. coresolution) in $\A$. 
We denote by $\Gproj\A$ (resp. $\Ginj\A$) the subcategory of $\A$ consisting of Gorenstein projective (resp. injective) objects in $\A$. When $\A=\Mod R$ for a ring $R$, we simply denote it by $\GProj R$ (resp. $\GInj R$). When $\A=\mod R$ for a right noetherian ring $R$, we simply denote it by $\Gproj R$ (resp. $\Ginj R$). Note that if $R$ is commutative noetherian ring, the category $\Gproj R$ coincides with the category of {\em totally reflexive} $R$-modules, denoted by $\tref(R)$; see \cite[Theorem 4.1.4]{Christensen t}.
\item
For an object $M$ in $\D^{-}(\A)$ (resp. $\D^{+}(\A)$), we denote by $\Gpd_{\A}M$ (resp. $\Gid_{\A}M$) the $\Gproj\A$-resolution (resp. $\Ginj\A$-coresolution) dimension of $M$ (see Definition \ref{def of resoldim} for the definition of (co)resolution dimension). When $\A=\Mod R$ for a ring $R$, we simply denote it by $\Gpd_{R}M$ (resp. $\Gid_{R}M$). When $\A=\mod R$ for a right noetherian ring $R$, we simply denote it by $\Gdim_{R}M$.
\end{enumerate}
\end{dfn}
\begin{rmk}\label{rmk for Gproj}
For an abelian category with enough projective objects, the category $\Gproj\A$ is a resolving subcategory of $\A$ by \cite[Remark 3.4 (3)]{Matsui Takahashi}. Dually, applying this result to the opposite category of an abelian category $\B$ with enough injective objects, the category $\Ginj\B$ is a coresolving subcategory of $\B$.
\end{rmk}
Now we have reached the main result of this section.
\begin{cor}\label{Gproj inj}
Let $\A$ be an abelian category.
\begin{enumerate}[\rm(1)]
\item 
Assume that $\A$ has enough projective objects. Then for all nonzero objects $M$ in $\D^{b}(\A)$, the following inequalities hold.
\begin{center}
$\level_{\D^{b}(\A)}^{\proj\A}M\geq \pd_{\A}M+\inf M +1$,
\vskip.5\baselineskip
$\level_{\D^{b}(\A)}^{\Gproj\A}M\geq \Gpd_{\A}M+\inf M +1$.
\end{center}
\item 
Assume that $\A$ has enough injective objects. Then for all nonzero objects $M$ in $\D^{b}(\A)$, the following inequalities hold.
\begin{center}
$\level_{\D^{b}(\A)}^{\inj\A}M\geq \id_{\A}M-\sup M +1$,
\vskip.5\baselineskip
$\level_{\D^{b}(\A)}^{\Ginj\A}M\geq \Gid_{\A}M-\sup M +1$.
\end{center}
\end{enumerate}
\end{cor}
\begin{proof}
(1) It is easy to see that $\proj\A$ is a resolving subcategory of $\A$ satisfying the condition $(\mathbf{A})$. Combining Remark \ref{rmk for Gproj} with and the definition of Gorenstein projective object in $\A$, we can easily see that $\Gproj\A$ is a resolving subcategory of $\A$ satisfying the condition $(\mathbf{A})$. Hence by Theorem \ref{main thm llevel}, the assertions hold.
(2) is the dual of (1), so we omit the proof.
\end{proof}
By applying the above results to the category of modules over a ring, we immediately recover \cite[Theorem 2.1]{AGMSV} and \cite[Theorem 3.3]{Awadalla Marley}.
\begin{cor}[Altmann--Grifo--Monta\~{n}o--Sanders--Vu, and Awadalla--Marley]
Let $R$ be a ring.
\begin{enumerate}[\rm(1)]
\item 
For a nonzero complex $M$ of left $R$-modules, one has
\begin{center}
$\level_{\D^{b}(\Mod R)}^{\Proj R}M \geq \pd_{R}M + \inf M +1$.
\end{center}
\item
Assume that $R$ is a commutative ring.
Then for all nonzero objects $M$ in $\D^{b}(\Mod R)$, one has
\begin{center}
$\level_{\D^{b}(\Mod R)}^{\GProj R}M \geq \Gpd_{R}M + \inf M +1$.
\end{center}
Moreover, if $R$ is noetherian and $M$ is in $\D^{b}(\mod R)$, then one has
\begin{center}
$\level_{\D^{b}(\mod R)}^{\tref(R)}M \geq \Gdim_{R}M + \inf M +1$.
\end{center}
\end{enumerate}
\end{cor}
\begin{proof}
In Corollary \ref{Gproj inj}, by taking $\Mod R$ as $\A$, the first two assertions follow. If $R$ is noetherian, by taking $\mod R$ as $\A$, the last assertion holds.
\end{proof}
\section{(Gorenstein) $C$-projective/injective levels}
In this section, we apply Theorem \ref{main thm llevel} to the category of modules over a commutative noetherian ring. Let us first recall the notion of a semidualizing module.
\begin{dfn}
Let $R$ be a commutative noetherian ring.
A finitely generated $R$-module $C$ is {\em semidualizing} if the following conditions hold.
\begin{enumerate}[\rm(1)]
\item
The natural homomorphism
$R \to \End_{R}(C)$ is an isomorphism.
\item
One has $\Ext_{R}^{i}(C,C)=0$ for all $i>0$.
\end{enumerate}
Note that $C$ is a semidualizing $R$-module if and only if the natural morphism $R\to\RHom_{R}(C,C)$ is an isomorphism in $\D(\mod R)$.
\end{dfn}
\begin{rmk}
The ring $R$ is a semidualizing $R$-module. If $R$ is a Cohen--Macaulay local ring with a canonical module $\omega$, then $\omega$ is a semidualizing $R$-module.
\end{rmk}
Now, we recall the notion of Gorenstein projective/injective modules with respect to a fixed semidualizing module.
\begin{dfn}\label{def of Gcproj}
Let $R$ be a commutative noetherian ring, and $C$ a semidualizing $R$-module.
\begin{enumerate}[\rm(1)]
\item
Let $M$ be a module in $\Mod R$ (resp. $\mod R$).
A {\em complete $C$-resolution} of $M$ in $\Mod R$ (resp. $\mod R$) is an exact sequence
\begin{center}
$X=\left(\cdots\to P^{-1}\to P^{0}\to C\otimes_{R}Q^{1}\to C\otimes_{R}Q^{2}\to\cdots \right)$
\end{center}
such that $P^{i}$, $Q^{i}\in \Proj R$ (resp. $\proj R$) for all $i$, and the complex $\Hom_{R}(X,C\otimes_{R}Q)$ is exact for all $Q$ in $\Proj R$ (resp. $\proj R$), and $M$ is isomorphic to $\cobound^{0}\!X$. Dually, a {\em complete $C$-coresolution} of $M$ in $\Mod R$ is an exact sequence
\begin{center}
$Y=\left(\cdots\to \Hom_{R}(C,J^{-2})\to \Hom_{R}(C,J^{-1}) \to I^{0}\to I^{1}\to \cdots \right)$
\end{center}
such that $I^{i}$, $J^{i}\in \Inj R$ for all $i$, and the complex $\Hom_{R}(\Hom_{R}(C,J),Y)$ is exact for all $J$ in $\Inj R$, and $M$ is isomorphic to $\cycle^{0}\!Y$.
\item
A module $M$ in $\Mod R$ is {\em Gorenstein $C$-projective} (resp. {\em Gorenstein $C$-injective}) if $M$ has a complete $C$-resolution (resp. $C$-coresolution) in $\Mod R$.  
We denote by $\G_{C}\Proj R$ (resp. $\G_{C}\Inj R$) the subcategory of $\Mod R$ consisting of Gorenstein $C$-projective (resp. $C$-injective) modules in $\Mod R$. In addition, we denote by $\tref_{C}(R)$ the subcategory of $\mod R$ consisting of Gorenstein $C$-projective modules in $\mod R$.
\item
For an object $M$ in $\D^{-}(\Mod R)$ (resp. $\D^{+}(\Mod R)$), we denote by $\G_{C}\pd_{R}M$ (resp. $\G_{C}\id_{R}M$) the $\G_{C}\Proj R$-resolution (resp. $\G_{C}\Inj R$-coresolution) dimension of $M$. When $M$ is a module in $\mod R$, we denote by $\G_{C}\dim_{R}M$ the $\tref_{C}(R)$-resolution dimension of $M$.
\end{enumerate}
\end{dfn}
\begin{rmk}\label{rmk for Gcproj}
\begin{enumerate}[\rm(1)]
\item 
The subcategory of $\mod R$ consisting of {\em totally $C$-reflexive} $R$-modules coincides with $\tref_{C}(R)$; see \cite[Theorem 4.1.4]{Christensen t} for instance. Note that when $R$ is a Cohen--Macaulay local ring with a canonical module $\omega$, the category $\tref_{\omega}(R)$ coincides with the category $\CM(R)$ of maximal Cohen--Macaulay $R$-modules.
\item
The subcategory $\G_{C}\Proj R$ (resp. $\tref_{C}(R)$) is a resolving subcategory of $\Mod R$ (resp. $\mod R$) by \cite[Theorem 2.8]{White}. Moreover, $\G_{C}\Proj R$ (resp. $\tref_{C}(R)$) satisfies the condition $(\mathbf{A})$ in Definition \ref{condition A} for the abelian category $\Mod R$ (resp. $\mod R$) by \cite[Proposition 2.9]{White}. 
Dually, the subcategory $\G_{C}\Inj R$ is a coresolving subcategory of $\Mod R$ satisfying the condition $(\mathbf{A}^{*})$.
\end{enumerate}
\end{rmk}
Combining Theorem \ref{main thm llevel} with the above Remark, we obtain the following result, which is the main result of the first half of this section.
\begin{cor}\label{Gcproj inj}
Let $R$ be a commutative noetherian ring, and $C$ a semidualizing  $R$-module.
\begin{enumerate}[\rm(1)]
\item 
For all nonzero objects $M$ in $\D^{b}(\Mod R)$, the following inequalities hold.
\begin{center}
$\level_{\D^{b}(\Mod R)}^{\G_{C}\Proj R}M\geq \G_{C}\pd_{R}M+\inf M +1$,
\vskip.5\baselineskip
$\level_{\D^{b}(\Mod R)}^{\G_{C}\Inj R}M\geq \G_{C}\id_{R}M-\sup M +1$.
\end{center}
\item
For all nonzero objects $M$ in $\D^{b}(\mod R)$, the following inequality holds.
\begin{center}
$\level_{\D^{b}(\mod R)}^{\tref_{C}(R)}M\geq \G_{C}\dim_{R}M+\inf M +1$.
\end{center}
\item
Assume that $R$ is a Cohen--Macaulay local ring with a canonical module $\omega$. Then for all nonzero objects $M$ in $\D^{b}(\mod R)$, the following inequality holds.
\begin{center}
$\level_{\D^{b}(\mod R)}^{\CM(R)}M\geq \G_{\omega}\dim_{R}M+\inf M +1$,
\end{center}
where $\CM(R)$ is the subcategory of $\mod R$ consisting of maximal Cohen--Macaulay $R$-modules.
\end{enumerate}
\end{cor}
Next, we establish a lower bound for the $C$-projective levels of a complex through Foxby equivalence. We first review the definitions of Auslander and Bass classes with respect to $C$.
\begin{dfn}
Let $R$ be a commutative noetherian ring, and $C$ a fixed semidualizing $R$-module.
\begin{enumerate}[\rm(1)]
\item
For an object $X$ in $\D^{b}(\Mod R)$, we denote by $\gamma_{X}$ the composition of the canonical morphisms $X\xrightarrow{\cong}\RHom_{R}(C,C)\Lotimes_{R}X\to\RHom(C,C\Lotimes_{R}X)$ in $\D(\Mod R)$.
Dually, for an object $Y$ in $\D^{b}(\Mod R)$, we denote by $\xi_{Y}$ the composition of the canonical morphisms $C\Lotimes_{R}\RHom_{R}(C,Y)\to \RHom_{R}(\RHom_{R}(C,C),Y) \xrightarrow{\cong}Y$ in $\D(\Mod R)$.
\item
We define the {\em Auslander class with respect to} $C$, denoted by $\A_{C}(R)$, the subcategory of $\D^{b}(\Mod R)$ consisting of complexes $M$ satisfying that $C\Lotimes_{R}M$ belongs to $\D^{b}(\Mod R)$, and $\gamma_{M}$ is an isomorphism.
Dually, we define the {\em Bass class with respect to} $C$, denoted by $\B_{C}(R)$, the subcategory of $\D^{b}(\Mod R)$ consisting of complexes $M$ satisfying that $\RHom_{R}(C,M)$ belongs to $\D^{b}(\Mod R)$, and $\xi_{M}$ is an isomorphism.
\end{enumerate}
\end{dfn}
\begin{rmk}
The Auslander class $\A_{C}(R)$ and Bass class $\B_{C}(R)$ are thick subcategories of $\D^{b}(\Mod R)$. Similarly, $\A_{C}(R)\cap\D^{b}(\mod R)$ and $\B_{C}(R)\cap\D^{b}(\mod R)$ are thick subcategories of $\D^{b}(\mod R)$.
\end{rmk}
Note that every semidualizing module is a semidualizing complex with amplitude zero (terminology is based on \cite{Christensen}). Hence, the following holds.
\begin{lem}\cite[(4.6), (4.8), (4.11)]{Christensen}\label{Chris 4}
Let $R$ be a commutative noetherian ring, and $C$ a semidualizing $R$-module.
\begin{enumerate}[\rm(1)]
\item
There exists a triangle equivalence of categories between $\A_{C}(R)$ and $\B_{C}(R)$ as follows.
\begin{equation*}
\xymatrix@R=18pt@C=100pt{
\A_{C}(R)\ar@{}[d]|{\bigcup} \ar@<0.5ex>[r]^{C\Lotimes_{R}-}  & \B_{C}(R) \ar@{}[d]|{\bigcup}\ar@<0.5ex>[l]^{\RHom_{R}(C,-)} \\
\A_{C}(R)\cap\D^{b}(\mod R) \ar@<0.5ex>[r]^{C\Lotimes_{R}-} & \B_{C}(R)\cap\D^{b}(\mod R) \ar@<0.5ex>[l]^{\RHom_{R}(C,-)}
}
\end{equation*}
In particular, for all objects $X$ in $\D^{b}(\Mod R)$, the following equivalences hold.
\begin{gather*}
X\in \A_{C}(R) \iff C\Lotimes_{R}X\in\B_{C}(R), \\
X\in \B_{C}(R) \iff \RHom_{R}(C,X)\in \A_{C}(R). 
\end{gather*}
Similarly, for all objects $X$ in $\D^{b}(\mod R)$, the following equivalences hold.
\begin{gather*}
X\in \A_{C}(R)\cap\D^{b}(\mod R) \iff C\Lotimes_{R}X\in\B_{C}(R)\cap\D^{b}(\mod R), \\
X\in \B_{C}(R)\cap\D^{b}(\mod R) \iff \RHom_{R}(C,X)\in \A_{C}(R)\cap\D^{b}(\mod R). 
\end{gather*}
\item
Let $X$ be an object in $\D^{b}(\Mod R)$. Then the following equalities hold.
\begin{itemize}
\item
$\sup C\Lotimes_{R}X=\sup X$.
\item
$\inf \RHom_{R}(C,X)=\inf X$.
\end{itemize}
\item
Let $X$ be an object in $\A_{C}(R)$, and $Y$ an object in $\B_{C}(R)$. Then the following equalities hold.
\begin{itemize}
\item
$\inf C\Lotimes_{R}X=\inf X$.
\item
$\sup \RHom_{R}(C,Y)=\sup Y$.
\end{itemize}
\end{enumerate}
\end{lem}
Next, we recall the notion of additive closures in the category of modules over a ring.
\begin{dfn}
Let $R$ be a ring, and $\X$ a subcategory of $\Mod R$. We denote by $\Add\X$ the subcategory of $\Mod R$ consisting of direct summands of direct sums of modules in $\X$.
\end{dfn}
\begin{rmk}\label{rmk for Add}
Let $R$ be a commutative noetherian ring, and $\X$ a subcategory of $\Mod R$ (resp. $\mod R$). Let $C$ be a semidualizing $R$-module.
\begin{enumerate}[\rm(1)]
\item
The subcategory $\Add\X$ (resp. $\add\X$) is the smallest subcategory of $\Mod R$ (resp. $\mod R$) that contains $\X$ and is closed under direct summands and direct sums (resp. finite direct sums) in $\Mod R$ (resp. $\mod R$).
\item
The subcategory $\Add R = \Add\{R\}$ (resp. $\add R = \add\{R\}$) coincides with $\Proj R$ (resp. $\proj R$).
\item
The Auslander class $\A_{C}(R)$ (resp. $\A_{C}(R)\cap\D^{b}(\mod R)$) contains the class of $R$-complexes with finite $\Add R$ (resp. $\add R$)-resolution dimension. On the other hand, the Bass class $\B_{C}(R)$ (resp. $\B_{C}(R)\cap\D^{b}(\mod R)$) contains the class of $R$-complexes with finite $\Add C$ (resp. $\add C$)-resolution dimension.
\item
The modules in $\Add C = \Add\{C\}$ (resp. $\add C = \add\{C\}$) are of the form $C\otimes_{R}P$, where $P\in\Proj R$ (resp. $P\in\proj R$). 
Indeed, let $P\in \Proj R$. Then $P$ is a direct summand of free $R$-module $R^{\oplus I}$ for some set $I$. Hence $C\otimes_{R}P$ is a direct summand of $C\otimes_{R}R^{\oplus I}\cong C^{\oplus I}$. Conversely, let $X\in \Add C$. Then $X$ is a direct summand of $C^{\oplus I}$ for some set $I$. Since $X$ belongs to $\B_{C}(R)$, one has $X\cong C\otimes_{R}\Hom_{R}(C,X)$. On the other hand, $\Hom_{R}(C,X)$ is a direct summand of $\Hom_{R}(C,C^{\oplus I})\cong R^{\oplus I}$. Thus, $P=\Hom_{R}(C,X)$ is a projective $R$-module. Hence, we have $X\cong C\otimes_{R}P$.
The same holds in the case of $\add C$.
\end{enumerate}
\end{rmk}
To establish Proposition \ref{level foxby}, we provide the following lemma.
\begin{lem}\label{equiv of C}
Let $R$ be a commutative noetherian ring, and $C$ a semidualizing $R$-module. Then the functor $C\otimes_{R}- : \Mod R\to \Mod R$ (resp. $C\otimes_{R}- : \mod R \to \mod R$) induces an equivalence of categories $C^{*}(\Add R) \to C^{*}(\Add C)$ (resp. $C^{*}(\add R) \to C^{*}(\add C)$) for all $*=+,-,b,\emptyset$. Moreover, let $I$ be a subset of $\mathbb{Z}$. Then for each object $X\in C(\Add C)$ (resp. $C(\add C)$) with $X^{i}=0$ for all $i\in I$, there exists an object $P\in C(\Add R)$ (resp. $C(\add R)$) such that $P^{i}=0$ for all $i\in I$ and $C\otimes_{R}P\cong X$ as $R$-complexes.
\end{lem}
\begin{proof}
To prove the equivalence of categories, it suffices to show that the induced functor $C\otimes_{R}- : \Add R \to \Add C$ is an equivalence of categories. By virtue of Remark \ref{rmk for Add} (4), the functor $C\otimes_{R}-$ is essentially surjective. Let $P$, $Q$ be modules in $\Proj R$. Then we have the following commutative diagram:
\begin{equation*}
\xymatrix@R=20pt@C=70pt{
\Hom_{R}(P,Q) \ar[r]^{C\otimes_{R}-}
& \Hom_{R}(C\otimes_{R}P,C\otimes_{R}Q) 
\\
& \Hom_{R}(P, \Hom_{R}(C, C\otimes_{R}Q)) \ar[u]^{\cong}
\\
& \Hom_{R}(P,Q). \ar[u]^{\cong} \ar@{=}[luu]
}
\end{equation*}
Hence, the functor $C\otimes_{R}-$ is fully faithful, and we obtain the desired equivalence of categories.
The latter assertion follows from the equivalence $C\otimes_{R}P=0 \iff P=0$, which holds for all $P\in \Proj R$.
The same holds in the case of $\add C$.
\end{proof}
The following proposition implies that the level and resolution dimension of a complex are preserved under Foxby equivalence.
\begin{prop}\label{level foxby}
Let $R$ be a commutative noetherian ring, and $C$ a semidualizing $R$-module.
\begin{enumerate}[\rm(1)]
\item
Let $\X$ be a subcategory of $\A_{C}(R)$, and $\Y$ a subcategory of $\B_{C}(R)$.
Let $M$ be an object in $\D^{b}(\Mod R)$.
\begin{enumerate}[\rm(a)]
\item
For all nonnegative integers $n$, the equivalence
\begin{center}
$M\in{\langle\X\rangle}^{\D^{b}(\Mod R)}_{n} \iff C\Lotimes_{R}M\in {\left\langle C\Lotimes_{R}\X\right\rangle}^{\D^{b}(\Mod R)}_{n}$
\end{center}
holds. In particular, one has 
\begin{center}
$\level_{\D^{b}(\Mod R)}^{\X}M = \level_{\D^{b}(\Mod R)}^{C\Lotimes_{R}\X}\C\Lotimes_{R}M$.
\end{center}
\item
For all nonnegative integers $n$, the equivalence
\begin{center}
$M\in{\langle\Y\rangle}^{\D^{b}(\Mod R)}_{n} \iff \RHom_{R}(C,M)\in {\left\langle\RHom_{R}(C,\Y)\right\rangle}^{\D^{b}(\Mod R)}_{n}$
\end{center}
holds. In particular, one has 
\begin{center}
$\level_{\D^{b}(\Mod R)}^{\Y}M = \level_{\D^{b}(\Mod R)}^{\RHom_{R}(C,\Y)}\RHom_{R}(C,M)$.
\end{center}
\end{enumerate}
\item
Let $\X$ be a subcategory of $\A_{C}(R)\cap\D^{b}(\mod R)$, and $\Y$ a subcategory of $\B_{C}(R)\cap\D^{b}(\mod R)$.
Let $M$ be an object in $\D^{b}(\mod R)$.
\begin{enumerate}[\rm(a)]
\item
For all nonnegative integers $n$, the equivalence
\begin{center}
$M\in{\langle\X\rangle}^{\D^{b}(\mod R)}_{n} \iff C\Lotimes_{R}M\in {\left\langle C\Lotimes_{R}\X\right\rangle}^{\D^{b}(\mod R)}_{n}$
\end{center}
holds. In particular, one has 
\begin{center}
$\level_{\D^{b}(\mod R)}^{\X}M = \level_{\D^{b}(\mod R)}^{C\Lotimes_{R}\X}\C\Lotimes_{R}M$.
\end{center}
\item
For all nonnegative integers $n$, the equivalence
\begin{center}
$M\in{\langle\Y\rangle}^{\D^{b}(\mod R)}_{n} \iff \RHom_{R}(C,M)\in {\left\langle\RHom_{R}(C,\Y)\right\rangle}^{\D^{b}(\mod R)}_{n}$
\end{center}
holds. In particular, one has 
\begin{center}
$\level_{\D^{b}(\mod R)}^{\Y}M = \level_{\D^{b}(\mod R)}^{\RHom_{R}(C,\Y)}\RHom_{R}(C,M)$.
\end{center}
\end{enumerate}
\item
Let $M$ be an object in $\D^{b}(\Mod R)$. Then the following hold.
\begin{enumerate}[\rm(a)]
\item
One has $\pd_{R}M=\Add C\resoldim C\Lotimes_{R}M$.
\item
One has $\Add C\resoldim M=\pd_{R}\RHom_{R}(C,M)$.
\end{enumerate}
\item
Let $M$ be an object in $\D^{b}(\mod R)$. Then the following hold.
\begin{enumerate}[\rm(a)]
\item
One has $\pd_{R}M=\add C\resoldim C\Lotimes_{R}M$.
\item
One has $\add C\resoldim M=\pd_{R}\RHom_{R}(C,M)$.
\end{enumerate}
\end{enumerate}
\end{prop}
\begin{proof}
The proofs of (1) and (3) are similar to those of (2) and (4). Therefore, we only provide the proofs for (1) and (3).

(1) Since $\A_{C}(R)$ and $\B_{C}(R)$ are thick subcategories of $\D^{b}(\Mod R)$, the following inclusion relations hold for all $n\geq 0$.
\begin{center}
${\langle\X\rangle}^{\D^{b}(\Mod R)}_{n}\subseteq \A_{C}(R)$,
\qquad 
${\left\langle C\Lotimes_{R}\X\right\rangle}^{\D^{b}(\Mod R)}_{n}\subseteq \B_{C}(R)$,
\vskip.5\baselineskip
${\langle\Y\rangle}^{\D^{b}(\Mod R)}_{n}\subseteq \B_{C}(R)$, 
\qquad
${\left\langle\RHom_{R}(C,\Y)\right\rangle}^{\D^{b}(\Mod R)}_{n}\subseteq \A_{C}(R)$.
\end{center}
From the above observations, together with the triangle equivalence in Lemma \ref{Chris 4} (1) and induction on $n$, the assertion follows.

(3) (a) We shall prove that $\pd_{R}M\geq\Add C\resoldim C\Lotimes_{R}M$. We may assume that $n=\pd_{R}M$ is an integer. Then we can take a $\Proj R$-resolution $P$ of $M$ of the following form:
\begin{center}
$P=\left(0\to P^{-n}\to\cdots\to P^{\sup M}\to0\right)$.
\end{center}
Hence, we have
\begin{align*}
C\Lotimes_{R}M
& \cong C\otimes_{R}P \\
& = \left(0\to C\otimes_{R}P^{-n}\to\cdots\to C\otimes_{R}P^{\sup M}\to0\right)=X.
\end{align*}
Since the equality $\sup M=\sup C\Lotimes_{R}M$ holds by Lemma \ref{Chris 4}, $X$ gives an $\Add C$-resolution of $C\Lotimes_{R}M$. Thus, the inequality $\Add C\resoldim C\Lotimes_{R}M\leq n$ holds. Next, we shall prove that $\pd_{R}M\leq\Add C\resoldim C\Lotimes_{R}M$. We may assume that $n=\Add C\resoldim C\Lotimes_{R}M$ is an integer. Then we can take an $\Add C$-resolution $X$ of $C\Lotimes_{R}M$ of the following form:
\begin{center}
$X=\left(0\to C^{-n}\to\cdots\to C^{\sup C\Lotimes_{R}M}\to0\right)$.
\end{center}
Note that $\sup C\Lotimes_{R}M=\sup M$ holds by Lemma \ref{Chris 4}. By virtue of Lemma \ref{equiv of C}, there exists a bounded complex $P$ such that $P^{i}\in \Add R$ for all $i$, and $C\otimes_{R}P$ is isomorphic to $X$ as $R$-complexes, and $P$ is of the following form:
\begin{center}
$P=\left(0\to P^{-n}\to\cdots\to P^{\sup M}\to0\right)$.
\end{center}
As $C\Lotimes_{R}M$ belongs to $\B_{C}$, $M$ is in $\A_{C}(R)$. Combining this with the fact that $P$ belongs to $\A_{C}(R)$, we have the following isomorphisms in $\D^{b}(\Mod R)$. 
\begin{align*}
M \cong \RHom_{R}(C,C\Lotimes_{R}M)
& \cong \RHom_{R}(C,X) \\
& \cong \RHom_{R}(C,C\otimes_{R}P) \\
& \cong \RHom_{R}(C,C\Lotimes_{R}P) \\
& \cong P.
\end{align*}
Thus, the inequality $\pd_{R}M\leq n$ holds.

(b) We shall prove that $\Add C\resoldim M\leq\pd_{R}\RHom_{R}(C,M)$. We may assume that $n=\pd_{R}\RHom_{R}(C,M)$ is an integer. Then we can take a $\Proj R$-resolution $P$ of $\RHom_{R}(C,M)$ of the following form:
\begin{center}
$P=\left(0\to P^{-n}\to\cdots\to P^{\sup \RHom_{R}(C,M)}\to0\right)$.
\end{center}
Since $\RHom_{R}(C,M)$ has finite $\Proj R$-resolution dimension, $M$ belongs to $\B_{C}(R)$. Thus one has $\sup \RHom_{R}(C,M)=\sup M$ by Lemma \ref{Chris 4} (3) . Now we have the following isomorphisms in $\D^{b}(\Mod R)$:
\begin{align*}
M
& \cong C\Lotimes_{R}\RHom_{R}(C,M) \\
& \cong C\otimes_{R}P \\
&  = \left(0\to C\otimes_{R}P^{-n}\to\cdots\to C\otimes_{R}P^{\sup M} \right).
\end{align*}
Hence, the inequality $\Add C\resoldim M\leq n$ holds. Next, we shall prove that $\Add C\resoldim M\geq\pd_{R}\RHom_{R}(C,M)$ We may assume that $n=\Add C\resoldim M$ is an integer. Then we can take an $\Add C$-resolution $X$ of $M$ of the following form:
\begin{center}
$X=\left(0\to C^{-n}\to\cdots\to C^{\sup M}\to0\right)$.
\end{center}
In a similar approach to the case of (a), there exists a bounded complex $P$ of projective $R$-modules such that $C\otimes_{R}P$ is isomorphic to $X$, and $P$ is of the following form:
\begin{center}
$P=\left(0\to P^{-n}\to\cdots\to P^{\sup M}\to0\right)$.
\end{center} 
Hence, we have the following isomorphisms in $\D^{b}(\Mod R)$.
\begin{align*}
\RHom_{R}(C,M)
& \cong \RHom_{R}(C,X) \\
& \cong \RHom_{R}(C,C\otimes_{R}P) \\
& \cong \RHom_{R}(C,C\Lotimes_{R}P) \\
& \cong P.
\end{align*}
Thus, the inequality $n\geq \pd_{R}\RHom_{R}(C,M)$ holds.
\end{proof}
The above result yields the following lower bound for the $C$-projective levels of a complex.
\begin{cor}\label{C level}
Let $R$ be a commutative noetherian ring, and $C$ a semidualizing $R$-module.
\begin{enumerate}[\rm(1)]
\item
Let $M$ be a nonzero object  in $\D^{b}(\Mod R)$.
Then the following inequality holds.
\begin{center}
$\level_{\D^{b}(\Mod R)}^{\Add C}M\geq \Add C\resoldim M +\inf M +1$.
\end{center}
\item
Let $M$ be a nonzero object  in $\D^{b}(\mod R)$.
Then the following inequality holds.
\begin{center}
$\level_{\D^{b}(\mod R)}^{\add C}M\geq \add C\resoldim M +\inf M +1$.
\end{center}
\item
Let $R$ be a $d$-dimensional Cohen--Macaulay local ring with a canonical module $\omega$, and $M$ a finitely generated nonzero $R$-module with finite injective dimension.
Then the following inequality holds.
\begin{center}
$\level_{\D^{b}(\mod R)}^{\add\omega}M\geq d-\depth M+1$.
\end{center}
That is, one has $M\notin{\left\langle\omega\right\rangle}^{\D^{b}(\mod R)}_{d-\depth M}$.
\end{enumerate}
\end{cor}
\begin{proof}
(1) The following holds:
\begin{align*}
\level_{\D^{b}(\Mod R)}^{\Add C}M
& = \level_{\D^{b}(\Mod R)}^{\Proj R}\RHom_{R}(C,M) \text{\quad(by \ref{level foxby} (1)(b))} \\
& \geq \pd_{R}\RHom_{R}(C,M)+\inf\RHom_{R}(C,M)+1 \text{\quad(by \ref{Gproj inj} (1))} \\
& = \Add C\resoldim M +\inf M +1. \text{\quad(by \ref{level foxby} (3)(b) and \ref{Chris 4} (2))}
\end{align*}
Thus, we obtain the desired inequality. (2) can be proven similarly to (1).

(3) By virtue of \cite[Proposition 1.2.9 and Exercise 3.3.28]{BH}, the $\add\omega$-resolution dimension of $M$ is $d-\depth M$. Hence, the assertion follows from (2).
\end{proof}
\section{Contravarianly finite resolving levels}
In this section, we apply Theorem \ref{main thm llevel} to the category of modules over a noetherian algebra.  We start by recalling the notion of  approximations.
\begin{dfn}\label{def of approx}
Let $\mathcal{C}$ be an additive category, and $\X$ a subcategory of $\mathcal{C}$. A morphism $f:X\to C$ in $\mathcal{C}$ with $X\in \X$ is {\em right} $\X${\em -approximation} of $C$ if the induced morphism $\Hom_{\mathcal{C}}(X',X)\to \Hom_{\mathcal{C}}(X',C)$ is surjective for all $X'\in \X$. A right $\X$-approximation $f:X\to C$ is {\em minimal right} $\X${\em -approximation} of $C$ if every morphism $g:X\to X$ satisfying $f=fg$ is an isomorphism. We say that $\X$ is {\em contravariantly finite} in $\mathcal{C}$ if every object $C$ in $\mathcal{C}$ has a right $\X$-approximation.
Dually, a {\em (minimal)} {\em left} $\X${\em -approximation}, and a {\em covariantly finite} subcategory are defined.
\end{dfn}
\begin{rmk}\label{rmk for approx}
\begin{enumerate}[\rm(1)]
\item 
Let $\mathcal{C}$ be a Krull--Schmidt category, and $\X$ an additive subcategory of $\mathcal{C}$. If a morphism $X\to C$ is a right $\X$-approximation of $C$, then there exists a minimal right $\X$-approximation $X'\to C$ of $C$.
\item
Let $R$ be a commutative noetherian ring, and $\A$ an $R$-linear abelian category satisfying that $\Hom_{\A}(X,Y)$ is a finitely generated $R$-module for all objects $X$, $Y\in \A$. Then for all objects $X$ in $\A$, the subcategory $\add X$ is both covariantly and contravariantly finite (see \cite[Remark 5.6]{Aihara Takahashi 2015} for instance).
\end{enumerate}
\end{rmk}
Next, we provide a sufficient condition for a contravariantly finite resolving subcategory of an abelian category to satisfy the condition $(\mathbf{A})$.
For an additive category $\mathcal{C}$ and its subcategory $\X$, we denote by $\Sub \X$ the subcategory of $\mathcal{C}$ consisting of objects $M$ such that there exists an monomorphism $M\to X$ with $X\in \add \X$.
\begin{prop}\label{cond A in KS}
Let $\A$ be a Krull--Schmidt abelian category with enough projective objects, and $\X$ a contravariantly finite resolving subcategory of $\A$. Assume that $\X$ is contained in $\Sub(\X^{\perp})$. Then $\X$ satisfies the condition $(\mathbf{A})$. Hence, we obtain the following inequality for all nonzero objects $M$ in $\D^{b}(\A)$.
\begin{center}
$\level_{\D^{b}(\A)}^{\X}M \geq \X \resoldim M + \inf M +1$.
\end{center}
\end{prop}
\begin{proof}
Let $X$ be an object in $\X$. By assumption, there exists an exact sequence $0\to X\to E \to Z\to 0$ in $\A$ such that $E$ belongs to $\X^{\perp}$. Take a minimal right $\X$-approximation $f:X'\to Z$ of $Z$. Since $\X$ contains $\proj\A$, the morphism $f$ is an epimorphism. We set $Y=\ker(f)$. By Wakamatsu's lemma (see \cite[Lemma 1.3]{AR} or \cite[Lemma 2.1.1]{Xu} for instance), one has $\Ext_{\A}^{1}(\X,Y)=0$. Since $\X$ is a resolving subcategory of $\A$, we have $Y\in\X^{\perp}$. Now, we obtain the following pullback diagram in $\A$. 
\begin{equation*}
\xymatrix{
&& 0 \ar[d] & 0 \ar[d] & \\
&& X \ar[d] \ar@{=}[r] & X \ar[d] & \\
0 \ar[r] & Y \ar@{=}[d]\ar[r] & Y' \ar[d]\ar[r] & E \ar[d]\ar[r] & 0 \\
0 \ar[r] & Y \ar[r] & X' \ar[r]^{f}\ar[d] & Z \ar[r]\ar[d] & 0 \\
&& 0 & 0 & 
}
\end{equation*}
From the central vertical and horizontal exact sequences, one has $Y'\in \X\cap\X^{\perp}$. Thus, the exact sequence $0\to X\to Y'\to X'\to 0$ is the desired one.
\end{proof}
As a direct consequence of the above proposition, we obtain the following lower bound for the contravariantly finite resolving levels of a complex, which can also be obtained by combining Theorem \ref{main thm llevel} with \cite[Proposition 3.4]{AR}.
\begin{cor}\label{cor of AR}
Let $\Lambda$ be an artin algebra, and $\X$ a contravariantly finite resolving subcategory of $\mod \Lambda$. Then for all objects $M$ in $\D^{b}(\mod \Lambda)$, the following inequality holds.
\begin{center}
$\level_{\D^{b}(\mod \Lambda)}^{\X}M \geq \X \resoldim M + \inf M +1$.
\end{center}
\end{cor}
We end this section by providing an application for the resolving subcategory obtained as the orthogonal subcategory of a self-orthogonal object.
Let $\A$ be an abelian category. An object $T$ of $\A$ is {\em self-orthogonal} if one has $\Ext_{\A}^{i}(T,T)=0$ for all $i>0$.
The following proposition is due to Yuya Otake. The author is indebted to him for his offer to include it in this paper.
\begin{prop}\label{self-orth}
Let $R$ be a commutative noetherian ring, and $\Lambda$ a noetherian $R$-algebra. Let $T$ be a self-orthogonal module in $\mod \Lambda$, and we set $\X= {}^{\perp}T$. Assume that $\X$ is contained in $\Sub T$. Then $\X$ satisfies the condition $(\mathbf{A})$. Hence, we have the following inequality for all nonzero objects $M$ in $\D^{b}(\mod \Lambda)$.
\begin{center}
$\level_{\D^{b}(\mod \Lambda)}^{\X}M \geq \X \resoldim M + \inf M +1$.
\end{center}
\end{prop}
\begin{proof}
Let $X$ be a module in $\X$. Since $\add T$ is a covariantly finite subcategory of $\mod \Lambda$, there exists a left $\add T$-approximation $f:X\to T_{0}$ of $X$. By assumption, there exists a monomorphism $m:X\to T'$ with $T'\in \add T$, and $m$ factors through $f$. Hence $f$ is also a monomorphism, and we obtain the exact sequence $0\to X\xrightarrow{f} T_{0}\to X'\to 0$ in $\mod \Lambda$. Note that $\add T \subseteq \X\cap\X^{\perp}$. We shall show that $X'$ belongs to $\X$. Applying the functor $\Hom_{\Lambda}(-,T)$ to the above exact sequence, one can easily see that $\Ext_{\Lambda}^{i}(X',T)=0$ for all $i>1$, and we have the following exact sequence.
\begin{center}
$\Hom_{\Lambda}(T_{0},T)\xrightarrow{\Hom_{\Lambda}(f,T)} \Hom_{\Lambda}(X,T)\to \Ext_{\Lambda}^{1}(X',T) \to 0$.
\end{center}
Since $f$ is a left $\add T$-approximation of $X$, the induced morphism $\Hom_{\Lambda}(f,T)$ is surjective, and one has $\Ext_{\Lambda}^{1}(X',T)=0$. Thus, $X'\in {}^{\perp}T=\X$ and the exact sequence $0\to X\xrightarrow{f} T_{0}\to X'\to 0$ is the desired one.
\end{proof}
\begin{rmk}
Let $R$ be a commutative noetherian ring, and $\Lambda$ a noetherian $R$-algebra. Let $T$ be a self-orthogonal module in $\mod \Lambda$, and $\X$ a resolving subcategory of $\mod \Lambda$.
\begin{enumerate}[\rm(1)]
\item
Assume that $\X={}^{\perp}T$ and $\add T=\X\cap\X^{\perp}$. Then the converse of the above proposition also holds. Namely, if $\X$ satisfies the condition $(\mathbf{A})$, the one has $\X\subseteq \Sub T$.
\item
Assume that $\add T=\X\cap\X^{\perp}$ and $\X$ satisfies the condition $(\mathbf{A})$. Then one has $\X\subseteq {}^{\perp}T$. However, the equality does not hold in general (see \cite[Corollary 1.5]{KLOT} for instance). Note that the equality $\add T=\X\cap\X^{\perp}$ may not be considered a strong assumption from the viewpoint of \cite[p.261]{Happel}. 
\end{enumerate}
\end{rmk}
\begin{ac}
The author would like to thank his supervisor Ryo Takahashi for giving many thoughtful questions and helpful discussions. He also thanks Yuya Otake for his valuable comments.
\end{ac}

\end{document}